\documentclass[12pt]{article}
\usepackage{a4wide}

\usepackage[a4paper, margin=2.3cm]{geometry}
\usepackage{amsmath,amssymb,amsthm}
\usepackage{color}
\usepackage{parskip}
\usepackage{hyperref}
\usepackage{parskip}
\usepackage{setspace}
\usepackage{graphicx}

\newtheorem{theorem}{Theorem}[section]
\newtheorem{remark}{Remark}[section]

\newtheorem{lemma}[theorem]{Lemma}

\newtheorem*{definition*}{Definition}

\begin{document}
\title{On the Mattila-Sj\"olin distance theorem for product sets}

\author{
Doowon Koh\thanks{Department of Mathematics, Chungbuk National University. Email: {\tt koh131@chungbuk.ac.kr}}
\and 
Thang Pham\thanks{Department of Mathematics,  ETHZ Switzerland. Email: {\tt phamanhthang.vnu@gmail.com}}
  \and
Chun-Yen Shen \thanks{Department of Mathematics,  National Taiwan University. Email: {\tt cyshen@math.ntu.edu.tw}}}
\maketitle  
\date{}

\begin{abstract}
Let $A$ be a compact set in $\mathbb{R}$, and $E=A^d\subset \mathbb{R}^d$. We know from
the Mattila-Sj\"olin's theorem if $\dim_H(A)>\frac{d+1}{2d}$, then the distance set $\Delta(E)$ has non-empty interior. In this paper, we show that the threshold $\frac{d+1}{2d}$ can be improved whenever $d\ge 5$. 
\end{abstract}
\section{Introduction}
Let $E$ be a compact set in $\mathbb{R}^d$, we denote its distance set by $\Delta(E)$, namely, 
\[\Delta(E):=\{|x-y|\colon x, y\in E\}.\] 
The classical Falconer distance conjecture says that if the  Hausdorff dimension of $E$, denoted by $\dim_H(E),$ is greater than $\frac{d}{2}$, then the Lebesgue measure of the distance set $\Delta(E)$ is positive. 

In 1986, Falconer \cite{falconer} proved that if $\dim_H(E)>\frac{d+1}{2}$, then $\mathcal{L}^1(\Delta(E))>0$. For $d=2$, the best current result is due to Guth, Iosevich, Ou, and Wang \cite{guth}, more precisely, they showed that the condition $\dim_H(E)>\frac{5}{4}$ would be enough. This result has been recently extended to all even dimensions $d\ge 4$ by Du, Iosevich, Ou, Wang, and Zhang \cite{du}. In odd dimensions, the current thresholds are $\frac{3}{2}+\frac{3}{10}$ for $d=3$ by Du, Guth, Ou, Wang, Wilson, and Zhang \cite{Du1}, and $\frac{d}{2}+\frac{d}{4d-2}$ for $d\ge 5$ by Du and Zhang \cite{Du2}.

In another direction, Mattila and
Sj\"{o}lin \cite{MS} proved a stronger result, namely,  if $\dim_H(E)>\frac{d+1}{2}$, 
then the distance set not only has positive Lebesgue measure, but also contains an interval. This was done by showing that the distance measure has a continuous density. To the best of our knowledge, it is not known in the literature if the threshold $\frac{d+1}{2}$ is the best possible in order for $\Delta(E)$ to  have non-empty interior even for product sets. In this paper, we focus on the case when the set $E$ has Cartesian product structures. More precisely, let $E=A \times A \cdots \times A= A^d \subset \mathbb R^d$ for a compact set $A\subset \mathbb{R}$, it follows from Mattila and Sj\"{o}lin's theorem that if $\dim_H(A)>\frac{d+1}{2d}$, then $\Delta(A^d)$ has non-empty interior. We denote by $Int(\Delta(A^d))$ the set of the interior points of the distance set $\Delta(A^d).$ In the following theorems, we prove that the bound $\frac{d+1}{2d}$ can be improved whenever $d\ge 5.$ 

First, for $d\ge 10,$ we  give an improvement of the threshold $\frac{d+1}{2d}.$ 
More precisely, we obtain the following.
\begin{theorem}\label{theorem1}
Let $A\subset \mathbb{R}$ be compact and  $d\ge 10.$ Suppose that $\dim_H(A)>\frac{d-2}{2(d-4)}-\frac{40}{57(d-4)}=\frac{d+1}{2d}-\frac{23d-228}{114d(d-4)}$, then $Int(\Delta(A^d))\ne \emptyset$.
\end{theorem}

An improvement of the bound $\frac{d+1}{2d}$ for $d\ge 5$ is given as follows.
\begin{theorem}\label{theorem2}
Let $A\subset \mathbb{R}$ be compact. Suppose that $d\ge 5$ and $\dim_H(A)>\frac{3d-2}{6d-8}
=\frac{d+1}{2d}-\frac{d-4}{2d(3d-4)}$, then $Int(\Delta(A^d))\ne \emptyset$.
\end{theorem}

It follows from a direct computation that when dimension $d \ge 27$, the threshold in Theorem \ref{theorem1} is better than that in Theorem \ref{theorem2}. Hence, we obtain the following theorem by combining this fact with the result $\frac{d+1}{2d}$ from Mattila and Sj\"{o}lin's theorem.  
\begin{theorem} Let $A$ be a compact set in $\mathbb R.$ Then we have $Int(\Delta(A^d))\ne \emptyset$ provided that
$$ \dim_H(A) > \left\{ \begin{array}{ll} \frac{d+1}{2d} \quad &\mbox{if}~~ 2\le d\le 4,\\
\frac{d+1}{2d}-\frac{d-4}{2d(3d-4)}\quad &\mbox{if}~~ 5\le d\le 26,\\
\frac{d+1}{2d}-\frac{23d-228}{114d(d-4)}\quad &\mbox{if}~~ 27\le d. \end{array} \right.$$
\end{theorem}

\section{Proofs of Theorem \ref{theorem1} and Theorem \ref{theorem2}}
To prove Theorem \ref{theorem1}, we make use of the following lemmas. We first recall an important and well-known result, see Theorem 1.2.1 in \cite{Sog}. 
\begin{lemma}
Let $\sigma$ be the surface measure on a compact piece of a smooth convex surface in $\mathbb R^d$, $ d \ge 2$, with everywhere non-vanishing Gaussian curvature. Then
$$|\widehat{\sigma}(\xi)| \lesssim |\xi|^{- \frac{d-1}{2}}.$$
\end{lemma}

Given two compact sets $E, F \subset \mathbb R^d$, by scaling and translation we may assume $E, F \subset [0, 1]^d$. Given a parabolic equation
$$P(x)=x_1^2+\cdots+x_{d-1}^2+x_d,$$
and since we are concerned with the distances set $P(x-y)$ for $x \in E$ and $y \in F$,  we can always assume the paraboloid surfaces that we will work on throughout the paper are contained in the fixed bounded ball $B(0, R)$.  
Given a paraboloid surface, it is straightforward to check that it has everywhere non-vanishing Gaussian curvature. 
We now can prove the following lemma.

\begin{lemma}\label{lm1}
Let $E$ and $F$ be compact sets in $\mathbb{R}^d$ with $\dim_H(E)>\frac{d+1}{2}$ and $\dim_H(F)>\frac{d+1}{2}.$ Then the set $ P(E-F):=\{P(x-y)\colon x\in E, y\in F\}$ has non-empty interior, where 
\[P(x):=x_1^2+\cdots+x_{d-1}^2+x_d.\]
\end{lemma}
\begin{proof}
We first note that if $\dim_H(E)>\frac{d+1}{2}$ and $\dim_H(F)>\frac{d+1}{2}$, then as a consequence of \cite[Theorem 1.8]{es} for two sets,  the set $P(E-F)$ is of positive Lebesgue measure. We now show that it also has non-empty interior. 

Let $\mu_E$ and $\mu_F$ be probability measures in $\mathcal{M}(E)$ and $\mathcal{M}(F)$, \footnote{For a Borel set $B$ in $\mathbb R^d$,  we denote by $\mathcal{M}(B)$ the collection of Borel measures $\mu$ with a compact support $\mathtt{spt}(\mu)\subset B$ and $0< \mu(B)<\infty.$}respectively, such that for any ball $B(x, r)$, we have 
\[\mu_E(B(x, r))\lesssim r^{\frac{d+1}{2}},~ \mu_F(B(x, r))\lesssim r^{\frac{d+1}{2}},\]
and $I_{\frac{d+1}{2}}(\mu_E)<\infty, ~I_{\frac{d+1}{2}}(\mu_F)<\infty,$ 
where we recall that $I_s(\mu) :=\int\int |x-y|^{-s} d\mu(x)d\mu(y).$

We now construct a measure in $\mathcal{M}(P(E-F)).$  Let $J: \mathbb R^d\times \mathbb R^d \to \mathbb{R}$ be a postive function defined as follows: for $x=(x_1, \ldots, x_{d-1}, x_d), ~y=(y_1, \ldots, y_{d-1}, y_d)\in \mathbb R^d,$
\[ J(x, y)=\sqrt{1+4(x_1-y_1)^2+\cdots+4(x_{d-1}-y_{d-1})^2}.\]
Consider a measure $J(x,y) d\mu_E(x) d\mu_F(y)$ supported on $E\times F$ and define $\widetilde{P}(x,y)=P(x-y).$
Next we define the $P$-distance meaure $\nu_{E, F} \in \mathcal{M}(P(E-F))$ by 
\[\nu_{E,F}(B)=\int_{\widetilde{P}^{-1}(B)} J(x,y) d\mu_E(x)d\mu_F(y),\]
for Borel sets $B \subset \mathbb R.$

In other words, $\nu_{E,F}$ is the pushforward measure defined as the image of $J (\mu_E \times \mu_F)$ under the map $\widetilde{P}(x,y)=P(x-y).$ Thus, for any continuous function $\varphi$ on $\mathbb R,$ we have 
\[\int \varphi d\nu_{E, F}=\int \int \varphi(P(x-y))J(x, y)d\mu_E(x)d\mu_F(y).\]

It is not hard to show that the support of $\nu_{E, F}$ is contained in $P(E-F),$ which is an important necessary condition for $P(E-F)$ to have non-empty interior.  Indeed, since $\nu_{E, F}=(\widetilde{P})_*(J(x, y)\mu_E(x)\mu_F(y))$, we have $\mathtt{spt}(\nu_{E, F})\subset \widetilde{P}\left(\mathtt{spt}(J(\mu_E\times \mu_F))\right)= \widetilde{P}\left(\mathtt{spt}(\mu_E\times \mu_F)\right)\subset P(E-F)$.

Moreover, we have a simple fact that if both $\mu_E^i \rightarrow \mu_E$ and $\mu_F^i \rightarrow \mu_F$ weakly, then
$\nu_{E,F}^i \rightarrow \nu_{E,F}$ weakly. For compactly supported smooth functions $f_1, f_2$ on $\mathbb R^d$, we can also define the pushforward measure $\nu_{f_1,f_2}$ by
\[\int g d\nu_{f_1,f_2}=\int \int g(P(x-y))J(x, y) f_1(x)f_2(y)dxdy,\] 
for any continuous function $g$ on $\mathbb R$. However, since $f_1$ and $f_2$ are smooth and compactly supported, the measure
$\nu_{f_1, f_2}$ is actually a function defined by $$\nu_{f_1, f_2}(t) = \int (\sigma_t \ast f_2)f_1,$$
where  $\sigma_t, t\in \mathbb R,$  denotes the surface measure on the surface $P_t:=\{y\in \mathbb{R}^d\colon P(y)=t\}$. 

To see this, we first notice that for $y=(y', y_d)\in P_t,$ we have
$$d\sigma_t(y)=\sqrt{1+ |\bigtriangledown{\phi_t(y')}|^2}~ dy'=\sqrt{1+4|y'|^2}dy',$$
where $\phi_t(y')=t-|y'|^2$ for $y'\in \mathbb R^{d-1}.$ Next,
 we observe that for any continuous function $g$ with compact support in $\mathbb R$,
\begin{align*}\int g(t) \int (\sigma_t \ast f_2)(x)f_1(x)dxdt &= \int g(t) \int \int f_2(x-y) d\sigma_t(y) f_1(x) dx dt\\
&= \int g(t) \int \int f_2(x'-y',  x_d-\phi_t(y')) \sqrt{1+4|y'|^2} dy' f_1(x) dx dt,\end{align*}
which is by Fubini's theorem and a change of variables, $y_d=\phi_t(y'),$
$$= \int \int  g(P(y))f_2(x-y) \sqrt{1+4|y'|^2} dy f_1(x) dx, $$
which is by a change of variables, replacing $y$  by $x-y$,
$$= \int \int g(P(x-y)) f_2(y)\sqrt{1+4|(x'-y')|^2} dy f_1(x) dx,$$
which is by Fubini's theorem and definition of $J(x,y),$
$$ =\int \int g(P(x-y)) J(x,y)f_1(x)f_2(y) dx dy= \int g d\nu_{f_1, f_2}.$$ 

Now let $\Psi$ be a smooth and compactly supported function in $\mathbb R^d$ with $\int \Psi =1,$ and denote
$\Psi_{\epsilon}(x) = \epsilon^{-d}\Psi(\frac{x}{\epsilon})$ and $\mu_{\epsilon} = \Psi_{\epsilon}\ast \mu$. Therefore we have 
$\widehat{\mu_{\epsilon}}(\xi) = \widehat{\Psi}(\epsilon \xi)\widehat{\mu}(\xi) \rightarrow \widehat{\mu}(\xi)$ for all $\xi \in \mathbb R^d$.
Now 
$$\nu_{\mu_{E_{\epsilon}}, \mu_{F_{\epsilon}}} (t) = \int (\sigma_t \ast \mu_{F_{\epsilon}})\mu_{E_{\epsilon}}=\int \widehat{\sigma_t} \widehat{\mu_{F_{\epsilon}}}  \overline{\widehat{\mu_{E_{\epsilon}}}},$$
which is 
$$\int \widehat{\sigma_t}(\xi)    |\widehat{\Psi}(\epsilon \xi)|^2  \widehat{\mu_F}(\xi) \overline{\widehat{\mu_E}}(\xi)d\xi   .$$
Recall that $I_{\frac{d+1}{2}}(\mu_E) \sim \int |\xi|^{\frac{1-d}{2}} |\widehat{\mu_E}(\xi)|^2 d\xi<\infty$ and  $I_{\frac{d+1}{2}}(\mu_F) \sim \int |\xi|^{\frac{1-d}{2}} |\widehat{\mu_F}(\xi)|^2d\xi<\infty$. These facts together with 
$|\widehat{\sigma_t}(\xi)| \lesssim |\xi|^{- \frac{d-1}{2}}$ and H\"older's inequality and  Lebesgue's dominated convergence theorem imply that when $\epsilon \rightarrow 0$, we have
$$\int \widehat{\sigma_t} \widehat{\mu_{F_{\epsilon}}}  \overline{\widehat{\mu_{E_{\epsilon}}}} \rightarrow \int  \widehat{\sigma_t} \widehat{\mu_F}  \overline{\widehat{\mu_E}}.$$
Since we also have $\nu_{\mu_{E_{\epsilon}}, \mu_{F_{\epsilon}}} (t) $ converges weakly to $\nu_{E, F}$, we conclude that $\nu_{E, F}$ is a function and
 \begin{equation}
\nu_{E, F}(t)=\int \widehat{\mu_F}(\xi)\overline{\widehat{\mu_E}(\xi)}\widehat{\sigma_t}(\xi)d\xi.
\end{equation}
We now prove that the $P$-distance measure $\nu_{E, F}$ is continuous. In other words, we show that
\begin{align*}
 \nu_{E, F}(t+h)-\nu_{E, F}(t)&= \int \widehat{\mu_F}(\xi)\overline{\widehat{\mu_E}}(\xi)(\widehat{\sigma_{t+h}}(\xi)-\widehat{\sigma_t}(\xi))d\xi,\\
\end{align*}
which goes to $0$ as $h\to 0$. As a result the $P$-distance set must contain an interval which is what we want.  
Indeed, the continuity of $\nu_{E, F}$ will be derived by Lebesgue's dominated convergence theorem with the following conditions 
\begin{enumerate}
\item $\lim_{h\to 0}\left(\widehat{\sigma_{t+h}}(\xi)-\widehat{\sigma_t}(\xi)\right)=0$. 
\item $|\widehat{\sigma_t}(\xi)|\le c(t)|\xi|^{-\frac{d-1}{2}}$.
\item $\int \widehat{\mu_F}(\xi)\overline{\widehat{\mu_E}}(\xi)(\widehat{\sigma_{t+h}}(\xi)-\widehat{\sigma_t}(\xi))d\xi<\infty$.
\end{enumerate}

For the condition $(1)$,  $\widehat{\sigma_{t+h}}(\xi)-\widehat{\sigma_t}(\xi) =\int e^{-2\pi i x\cdot \xi} d(\sigma_{t+h}(x) - \sigma_t(x) ) $ which is equal to
$$\int (e^{-2\pi i h \xi_d} -1)e^{-2\pi i x\cdot \xi} d \sigma_t(x).$$
Therefore by Lebesgue's dominated convergence theorem, we see that  $$\lim_{h\to 0}\left(\widehat{\sigma_{t+h}}(x)-\widehat{\sigma_t}(x)\right)=0.$$

The condition $(2)$ follows from Lemma 2.1. To check the condition $(3)$, we apply the Cauchy-Schwarz inequality to obtain 
\begin{align*}
\left\vert \int \widehat{\mu_F}(\xi)\overline{\widehat{\mu_E}}(\xi)(\widehat{\sigma_{t+h}}(\xi)-\widehat{\sigma_t}(\xi))d\xi\right\vert&\le  \int |\widehat{\mu_E}(\xi)||\widehat{\mu_F}(\xi)||\widehat{\sigma_{t+h}}(\xi)-\widehat{\sigma_t}(\xi)|d\xi  \\
&\lesssim  \int |\widehat{\mu_E}(\xi)||\widehat{\mu_F}(\xi)||\xi|^{-\frac{d-1}{2}}d\xi  \\
&\lesssim I_{\frac{d+1}{2}}(\mu_E)^{\frac{1}{2}}\cdot I_{\frac{d+1}{2}}(\mu_F)^{\frac{1}{2}}<\infty.
\end{align*}
\end{proof}

\begin{remark}
After posting this paper to Arxiv about 4 months, Allan Greenleaf and Alex Iosevich recently informed us that Lemma 2.2 can also be proved by using the main result (Theorem 1.5) in \cite{GIT}. More precisely,  we can define $\Phi(x,y) : \mathbb R^d \times \mathbb R^d \rightarrow \mathbb R$ by
$$\Phi(x, y)= P(x-y)=\sum_{j=1}^{d-1} (x_j-y_j)^2 +(x_d-y_d) = |x'-y'| +(x_d-y_d).$$
Given $t_0 \in \mathbb R$, let $Z_{t_0}= \{\Phi(x,y)=t_0\}$. Then, we can solve $y_d$ in terms of $x', y'$ and $x_d$ that is $y_d= -t_0 +x_d +|x'-y'|$. Furthermore, 
$$\nabla \Phi(x', x_d, y', y_d)= (2(x'-y'), 1, -2(x'-y'), -1).$$
Hence, the canonical relation $C_{t_0}$ given by 
$$\{(x', x_d, 2(x'-y')\theta, \theta, y', y_d(\cdot), 2(x'-y')\theta, \theta): x \in \mathbb R^d, y' \in \mathbb R^{d-1}, \theta \in \mathbb R-0 \}$$
can be checked to be a canonical graph. Thus, its associated Radon transform adds $\frac{d-1}{2}$ derivatives on $L^2$ Sobolev spaces which in turn through Theorem 1.5 in \cite{GIT} gives that 
$P(E-F)$ contains an interval as long as $\dim_H E + \dim_H F > d+1$. We refer the interested reader to \cite{GIT} for more details.
\end{remark}

We proceed to prove our results. First, we also need the following simple lemma.
\begin{lemma}\label{lm3}
Let $X$ be a set in $\mathbb{R}$, and define $X^2:=\{x^2\colon x\in X\},$ $-X^2:=\{-x^2\colon x\in X\}.$ Then we have $$\dim_H(X)=\dim_H(X^2)=\dim_H(-X^2).$$
\end{lemma}
\begin{proof} Since  it is obvious that $\dim_H(X^2)=\dim_H(-X^2),$  we only need to show that $\dim_H(X)=\dim_H(X^2).$
We use a well-known fact that if $f$ is a bi-Lipschitz map from $\mathbb R^n \rightarrow \mathbb R^n$, then $f$ preserves the Hausdorff dimension.
Now given a set $X \subset \mathbb R$, and without loss of generality we may assume $X \subset (0, \infty)$. Let $X_n = X \cap (n, n+1]$ for $n=0, 1, \cdots$. 
It is clear that except for $n=0$, the function $f(x)=x^2$ is bi-Lipschitz on the set $X_n$. Therefore, we may assume our set $X \subset (0, 1],$ otherwise we are done. However we can consider 
$X_n= X \cap (\frac{1}{n}, 1]$ so that for each $n$ we have $\dim_H(f(X_n))=\dim_H(X_n)$. Finally  $\dim_H(X)=\dim_H(\cup X_n)=\sup_{n} \dim_H(X_n)=\sup_{n} \dim_H(f(X_n))=\dim_H(\cup f(X_n))=\dim_H(f(X))$ which gives the result.
\end{proof}

We now recall two lemmas below. The results of these two lemmas show that the Hausdorff dimension of the distance set $\Delta(\Omega)$ has a nontrivial lower bound if we only assume $\dim_H(\Omega) >1$, where $\Omega\subset \mathbb R^2.$ They also play an important role in proving  our main theorems.

The first lemma is due to Shmerkin, and the second is due to Liu.
\begin{lemma}[\cite{hai}]\label{lm2}
For $\Omega\subset \mathbb{R}^2$ with  $\dim_H(\Omega)>1$, then we have 
\[\dim_H(\Delta(\Omega))\ge \frac{40}{57}.\]
\end{lemma}

\begin{lemma}[\cite{liu}]\label{lm2'}
For $\Omega\subset \mathbb{R}^2$ with $\dim_H(\Omega)>1$, then we have 
\[\dim_H(\Delta(\Omega))\ge \min \left\{\frac{4}{3}\dim_H(\Omega)-\frac{2}{3}, ~1 \right\}.\]
\end{lemma}

To compare these two lemmas, we remark that when $1<\dim_H(\Omega)$ which is very close to 1, the lower bound in Lemma \ref{lm2} is better. Otherwise, the lower bound in Lemma \ref{lm2'} is better. For instance, when $\dim_H(E) > 1 +\frac{1}{38}$, the lower bound in Lemma \ref{lm2'} is better. More precisely, we have
$$ \dim_H(\Delta(\Omega))\ge \left\{\begin{array}{ll}  1 \quad &\mbox{if}~~\frac{5}{4}\le  \dim_H(\Omega) \\
\frac{4}{3}\dim_H(\Omega)-\frac{2}{3}\quad &\mbox{if}~~ \frac{39}{38} \le \dim_H(\Omega) \le \frac{5}{4}\\
\frac{40}{57} \quad &\mbox{if}~~ 1 < \dim_H(\Omega) \le \frac{39}{38}. \end{array}\right.$$

\begin{remark}\label{ReK} Let $\Omega=A\times A$ for some $A\subset \mathbb R.$ Since $\dim(A^2):=\dim(\Omega)\ge 2\dim(A),$   we can invoke both Lemma \ref{lm2} and Lemma \ref{lm2'}
whenever $\dim(A)>1/2.$
\end{remark}

We are ready to prove Theorems \ref{theorem1} and \ref{theorem2}.
\begin{proof}[Proof of Theorem \ref{theorem1}]
Suppose that $\dim_H(A)=s$. Recall that $$\Delta(A^2) =\{( (x_1-y_1)^2+(x_2-y_2)^2 )^{1/2} : x_1,x_2, y_1,y_2 \in A \},$$
and $$\Delta(A^2)^2 =\{ (x_1-y_1)^2+(x_2-y_2)^2 : x_1,x_2, y_1,y_2 \in A \}.$$

For $d\ge 5$, set $E=A^{d-4}\times \Delta(A^2)^2\subset \mathbb{R}^{d-3}$ and $F=A^{d-4}\times -\Delta(A^2)^2\subset \mathbb{R}^{d-3}$. 
One can check that $P(E-F)=\Delta(A^d)^2$. Thus, if $P(E-F)$ contains an interval, so does $\Delta(A^d)$. 
Since $E,F\subset \mathbb R^{d-3},$ 
it follows from Lemma \ref{lm1} that our problem reduces to proving  the conditions $\dim_H(E), \dim_H(F)>\frac{d-2}{2}$ under the our hypotheses that $d\ge 10$ and $s>\frac{d-2}{2(d-4)}-\frac{40}{57(d-4)}.$

Notice that $ s>1/2.$ As mentioned in Remark \ref{ReK},  we can use Lemma \ref{lm2} with $\Omega=A\times A.$
Now, using Lemmas \ref{lm3} and \ref{lm2}, we know that 
\[\dim_H(E), \dim_H(F)\ge s(d-4)+\dim_H(\Delta(A^2))\ge s(d-4)+\frac{40}{57}.\]
Hence, we need the condition that $s(d-4)+\frac{40}{57}> \frac{d-2}{2}.$ Namely,
 if $s>\frac{d-2}{2(d-4)}-\frac{40}{57(d-4)}$, then the theorem follows. 

\end{proof}

We proceed to prove Theorem \ref{theorem2} whose proof is almost identical with that of Theorem \ref{theorem1} except that we use Lemma \ref{lm2'} instead of Lemma \ref{lm2}.
\begin{proof}[Proof of Theorem \ref{theorem2}] We adopt the same notation as in the proof of Theorem \ref{theorem1}. The same ideas with Lemma \ref{lm2'} imply that for $s:=\dim(A) >1/2,$ 
\[\dim_H(E), \dim_H(F)\ge s(d-4)+\dim_H(\Delta(A^2) \ge s(d-4)+\min\{\frac{8s}{3}-\frac{2}{3},~ 1\}.\]
As before, we need to find certain conditions on $s$ and $d$ such that 
\begin{equation}\label{eq1K}
s(d-4)+\min\left\{\frac{8s}{3}-\frac{2}{3},~ 1\right\} >  \frac{d-2}{2}.
\end{equation}

\textbf{(Case 1)} Assume that $s\ge 5/8.$ Then $\min\left\{\frac{8s}{3}-\frac{2}{3},~ 1\right\}=1,$  and so the inequality \eqref{eq1K} holds if $d\ge 5$ and  $s>1/2.$ This implies that if $s\ge 5/8$ and $d\ge 5,$ then we obtain the desired result that  $Int(\Delta(A^d))\ne \emptyset.$

\textbf{(Case 2)} Assume that  $1/2< s < 5/8.$ Then $\min\left\{\frac{8s}{3}-\frac{2}{3},~ 1\right\}=\frac{8s}{3}-\frac{2}{3}$, and thus  the inequality \eqref{eq1K} is the same as 
$ s> \frac{3d-2}{6d-8}.$
Since $\frac{1}{2}< \frac{3d-2}{6d-8}$,  this implies that  if $\frac{3d-2}{6d-8}<s<\frac{5}{8},$ then  we 
get the desirable conclusion that $Int(\Delta(A^d))\ne \emptyset.$ Notice that we also need the condition on $d$ such that $ \frac{3d-2}{6d-8} <\frac{5}{8}$, namely,  $d\ge 5.$

By Case 1 and Case 2,  we deduce that if $d\ge 5$ and $ s>\frac{3d-2}{6d-8}$, then $Int(\Delta(A^d))\ne \emptyset.$ This completes the proof.
\end{proof}

\section*{Acknowledgments}

The authors would like to thank  Allan Greenleaf, Alex Iosevich and Bochen Liu  for useful comments and suggestions. 

Doowon Koh was supported by the National Research Foundation of Korea (NRF) grant funded by the Korea government (MIST) (No. NRF-2018R1D1A1B07044469). Thang Pham was supported by Swiss National Science Foundation grant P4P4P2-191067. Chun-Yen Shen was supported in part by MOST, through grant 108-2628-M-002-010-MY4.

\end{document}